\documentclass{article}
\usepackage{graphicx} 
\usepackage{amsmath}
\usepackage{lineno}
\usepackage{url}
\usepackage{hyperref}
\usepackage{amssymb}
\usepackage{amsthm}
\usepackage{amsfonts} 
\usepackage{bm}
\usepackage{cases}
\usepackage{xcolor}
\usepackage{color,soul}
\usepackage{nicefrac}
\usepackage{amsmath}
\newtheorem{theorem}{Theorem}
\newtheorem{proposition}{Proposition}
\newtheorem{definition}[proposition]{Definition}

\newtheorem{lemma}[proposition]{Lemma}

\newtheorem{corollary}[proposition]{Corollary}
\newtheorem{example}[proposition]{Example}

\def\bN{\mathbb N}
\def\Down{Down}
\def\eps{\varepsilon}
\def\e{\eps}
\def\cC{\mathcal C}
\def\cR{\mathcal R}
\def\cG{\mathcal G}
\def\cO{\mathcal O}
\def\cS{\mathcal S}
\def\F{F}
\def\ato{\succcurlyeq}
\def\oto{\succcurlyeq_\cC}
\def\otog{\succcurlyeq_{\cC_\cG}}
\def\tto{\succcurlyeq_\cS}
\def\ttog{\succcurlyeq_{\cS_\cG}}
\newcommand{\BF}{\bf\boldmath }
\newcommand{\beq}{\begin{linenomath}\begin{equation*}} 
\newcommand{\eeq}{\end{equation*}\end{linenomath}} 
\newcommand{\beqn}{\begin{linenomath}\begin{equation}} 
\newcommand{\eeqn}{\end{equation}\end{linenomath}} 
\usepackage[normalem]{ulem}
\newtheorem{thmX}{Theorem}

\title{Shadow chains and Conley chains\\ for continuous-time semiflows}
\author{Roberto De Leo, Jim Yorke}
\date{\today}

\begin{document}

\maketitle

\begin{abstract}
    In a recent series of articles we introduced the concept of ``stream of a semiflow''.
    A stream is a closed and transitive binary relation which extends the relation ``being on the orbit of'' and allows to encode the qualitative behavior of a semiflow into a direct graph.
    The most important stream of a semiflow is its chain stream, based on Charles Conley's chains.
    
    In those previous works we omitted several details and proofs on continuous-time semiflows.
    In the present work we complement those articles as follows:
    (i) we provide a full proof of the closedness and transitivity of the chain stream for continuous-time semiflows; 
    (ii) we introduce the concept of ``shadow chain'' for a continuous-time semiflow, based on the Anosov-Sinai-Bowen idea of pseudo-orbit.
    Shadow chains have the advantage that fit naturally with semiflows arising from differential equations.
    Our main result is that, although the shadow chain stream and the Conley chain stream are in general distinct as binary relations, they yield the same chain-recurrent set, the same nodes, and the same chain graph whenever the semiflow has strong compact dynamics. 
    
    While doing this, we also introduce an equivalent definition of recurrent point of a stream in terms of forward-orbit equivalence, which simplifies several arguments below, and we strengthen the definition of s-uniform continuity of a semiflow, fixing a gap in the proof of some important results when the space is not locally compact.
\end{abstract}

\section{Introduction}
In~\cite{Con72}, Charles Conley introduced the notion of $(\varepsilon, T)$-chain of a continuous-time semiflow $F$ on a metric space $(X, d)$. Given two points $x, y \in X$, an $(\varepsilon, T)$-chain from $x$ to $y$ is a finite sequence $x = c_0, c_1, \ldots, c_n = y$ of points of $X$ together with times $t_0, \ldots, t_{n-1} \geq T$ such that $d(c_{k+1}, F^{t_k}(c_k)) < \varepsilon$ for all $k$. 

In a recent series of articles~\cite{DY25,DLY26,DLY26b,ADY25} we have used this notion in the development of a theory of \emph{streams} of dynamical systems. 
A \emph{stream} of $F$ is a closed and transitive binary relation on $X$ that extends its orbit relation; its equivalence classes are called \emph{nodes}, and the induced directed graph provides a compact encoding of the qualitative behavior of $F$. 
The principal example is the \emph{chain stream}, defined by declaring $y$ downstream from $x$ whenever either $y$ lies in the orbit of $x$ or, for every $\varepsilon > 0$ and $T > 0$, there is an $(\varepsilon, T)$-chain from $x$ to $y$.
A point $x$ is \emph{chain-recurrent} if either it is fixed or there is a $y\neq x$ such that $x$ is downstream from $y$ and $y$ is downstream from $x$ (in the present article we will rather use a definition that we introduce in Definition~\ref{def: nodes and graph} and that we prove to be equivalent to the original in Lemma~\ref{lemma: OC1}); the chain-recurrent set together with the chain graph records, in essence, all of the recurrent and gradient-type structure of $F$ (see~\cite{DY25}).

The pivotal role of the chain stream is due to the following fact.
Assume for simplicity that $X$ is locally compact and say that $F$ has {\em compact dynamics} when $F$ has a {\em global attractor}, namely a compact set $\cG\subset X$ that attracts every compact set of $X$ under $F$ (see Definition~\ref{def: compact dynamics} for the precise statement).
This means that the chain stream is canonically associated to this kind of semiflows and so its nodes and graph are fully determined by $F$.
We proved in~\cite[Theorem~5.42]{DLY26} that, if $F$ has compact dynamics and its smallest stream has only countably many nodes (which happens in most semiflows from applications), then the chain stream of $F$ is the smallest stream of $F$.
In~\cite[Theorem~11]{DLY26b} we extended this result to the case when $X$ is not locally compact under the hypothesis that $F$ has some kind of uniform continuity close to the global attractor.
In that case, we say that $F$ has {\em strong compact dynamics}.
While writing this article we noticed that the definition given in~\cite{DLY26b} is actually not strong enough to prove the full claim; in Definition~\ref{def: strong compact dynamics} we provide a strengthened version that fixes this gap.

For discrete-time semiflows, our theory uses the $\varepsilon$-chain notion introduced by Bowen~\cite{Bow75} based on ideas by Anosov~\cite{Ano67} and Sinai~\cite{Sin72}: an $\varepsilon$-chain of a map $f$ is a finite sequence $x = c_0, \ldots, c_n = y$ with $d(c_{k+1}, f(c_k)) < \varepsilon$ for all $k$. The discrete chain stream is then defined by quantifying only over $\varepsilon$. Comparison of the two formulations reveals a structural asymmetry: the continuous-time definition involves both a spatial parameter $\varepsilon$ and a temporal parameter $T$, while the discrete-time definition involves only $\varepsilon$. 

In the present article we introduce a new type of chains for continuous-time semiflows, which we call {\em shadow chains}, that is a direct analogue of the definition used for discrete-time semiflows and so it is better suited, for instance, in proofs of properties that involve chains of both discrete-time and continuous-time semiflows. 
A shadow $\varepsilon$-chain from $x$ to $y$ is a piecewise continuous curve $\gamma : [0, T] \to X$ with $T \geq 1$, $\gamma(0) = x$, $\gamma(T) = y$, satisfying the uniform tracking condition
\beq
  d\bigl(\gamma(t + \tau),\, F^\tau(\gamma(t))\bigr) < \varepsilon
  \quad \text{for all } \tau \in [0, 1] \text{ and } t \in [0, T - \tau].
\eeq
This formulation is particularly natural for semiflows arising from differential equations: for an ODE $\dot{z} = g(z)$ on a Banach space, shadow chains can be produced as solutions of the perturbed equation $\dot{z} = g(z) + u(t)$ with a control function $u$ of small integral norm (see Example~\ref{ex: DE}). 
A special case of the shadow chain definition was already used in~\cite{ADY25} without an equivalence proof with the Conley chain; the present article supplies that proof.

Our main result (Theorem~\ref{thm:main}) is that, under the hypothesis of strong compact dynamics, 
a property satisfied by many systems arising in applications, including every semiflow on a compact space and the semiflows generated by dissipative reaction-diffusion PDEs, the shadow chain stream and the Conley chain stream yield the same chain-recurrent set, the same nodes, and the same chain graph. 
Notice that, although these structures coincide, the relations themselves are genuinely distinct: Section~\ref{sec: defs} exhibits a flow without compact dynamics for which there exist $x, y$ with $x \succcurlyeq_C y$ but $x \not\succcurlyeq_S y$. 
Moreover, the strong compact dynamics hypothesis cannot be dropped: a second example, given as the final remark of this article, displays a flow without compact dynamics for which $\cR_\cC \neq \cR_\cS$.

Along the way, the present article also fills three further gaps in the streams series. 

First, as already mentioned above, the notion of \emph{s-uniform continuity} we introduced in~\cite{DLY26b}, needed when the phase space is not locally compact, as for semiflows from dissipative PDEs, turns out to be insufficient for the proofs of certain results in that paper, including ones on which the present article depends. With Definition~\ref{def: strong compact dynamics} we strengthen this notion to the form actually needed, in turn strengthening the definitions of \emph{strong attractor} and \emph{strong compact dynamics} given in~\cite{DLY26b}. 

Second, in Proposition~\ref{prop: C and O U C pari son} we prove that the ``chain stream'' defined in~\cite{DY25,DLY26,DLY26b,ADY25} is really a stream, a fact we implicitly assumed in those works. 
In the same proposition, we accomplish the same for the shadow chain stream.
The argument to prove these properties is not immediate, as closedness for Conley chains requires a non-trivial refinement of chains so that all transition times lie in a compact interval, and closedness for shadow chains requires careful one-jump estimates at boundary times.

Third, while proving Proposition~\ref{prop: C and O U C pari son}, we clarify and discuss in full detail the following key point left implicit in our previous works. 
The binary relation $\cC$ given by ``$x\oto y$ if and only if, for every $\eps>0$ and $T>0$, there is an $(\eps,T)$-chain from $x$ to $y$'' is closed and transitive.
Except in trivial cases, $\cC$ is not a stream because it does not contain the orbit relation $\cO_F$ of the semiflow.
Nevertheless, we show in Proposition~\ref{prop: C and O U C pari son} that the chain stream is equal to $\cO_F\cup\cC$ and that the chain relation $\cC$ and the chain stream have the same nodes and the same graph.
We also show the analogue properties for the shadow chain relation and the shadow stream. 


As a byproduct, we prove for Conley and shadow chains an equivalent forward-orbit characterization of chain-recurrence (Lemma~\ref{lemma: OC1}): $x$ is Conley (resp. shadow) chain-recurrent if and only if the entire forward orbit $\{F^t(x) : t \geq 0\}$ is Conley (resp. shadow) chain-equivalent. 
This formulation is sometimes more convenient in proofs and is used repeatedly in the sequel.

\medskip
The paper is organized as follows. 
Section~\ref{sec: defs} introduces the basic definitions, gives the differential-equations example, proves closedness and transitivity of the chain relations (Proposition~\ref{prop: closed and transitive}), proves the coincidence of the nodes and graph of the chain relations with those of the corresponding chain streams (Proposition~\ref{prop: C and O U C pari son}), exhibits the example distinguishing $\oto$ from $\tto$, and recalls the notion of strong compact dynamics. 
Section~\ref{sec: main results} contains four preparatory lemmas, Theorem~\ref{thm: R_S subset G} (which reduces the shadow chain-recurrent structure to the global attractor), the proofs of the two inclusions $\cR_\cC \subset \cR_\cS$ and $\cR_\cS \subset \cR_\cC$, the main result (Theorem~\ref{thm:main}) and the example showing that the strong compact dynamics hypothesis is necessary.
%
\section{Shadow and Conley chains}
\label{sec: defs}
Throughout the article we denote by $(X,d)$  a metric space and by $F$
a \emph{semiflow} on $X$, namely $F$ is a continuous map $[0,\infty)\times X\to X$ such that $\F^0=\mathrm{id}_X$ and $\F^{t+s}=\F^t\circ \F^s$ for all $t,s\geq0$.

\medskip
In~\cite{Con72}, Conley introduced the following concept as a simplified version of the concept of prolongation used in the 1960s by Joseph Auslander to define generalized recurrence~\cite{Aus64}.
%
\begin{definition}[Conley, 1972~\cite{Con72}]
    For $\e>0$ and $T>0$, an {\BF$(\e,T)$-chain from $x$ to $y$} is a finite sequence of points $x=x_0,x_1,\dots,x_N=y$ in $X$, $N\geq1$, and times $t_i\ge T$ such that
    \[
    d(\F^{t_i}(x_i),x_{i+1}\big)<\e,\qquad i=0,\dots,N-1.
    \]
    Sometimes we refer to such an $(\e,T)$-chain as a {\bf Conley chain}.
\end{definition}
\begin{definition}
    We write {\BF $x\oto y$} if, for every $\e>0$ and $T>0$, there exists an $(\e,T)$-chain from $x$ to $y$ and
    we set 
    $$
    \text{\BF$\cC$}=\{(x,y):\;x\oto y\}\subset X\times X.
    $$
\end{definition}
Based on the concept of pseudo-orbit introduced in literature by Anosov, Bowen and Sinai, we propose the following alternate definition of chain.
%
\begin{definition}
    Let $F$ be a semiflow on $X$. 
    We say that a curve $\gamma:[0,T]\to X$ is
    {\BF $\eps$-close to the orbits of $F$} (or, for brevity, to $F$) if:
    \begin{enumerate}
        \item $T\geq1$;
        \item $\gamma$ is piecewise continuous;
        \item for every $\tau\in[0,1]$ and $t\in[0,T-\tau]$,
    \beqn
    \label{eq: S}
    d(\gamma(t+\tau),F^{\tau}(\gamma(t)))<\eps. 
    \eeqn
    \end{enumerate}

    Given $\eps>0$, an {\BF $\eps$-chain from $x$ to $y$} is a curve $\gamma:[0,T]\to X$ which is $\eps$-close to $F$ and such that 
    $\gamma(0)=x$ and $\gamma(T)=y$.
    Sometimes we refer to such $\e$-chain as a {\bf shadow chain}.
    When $y=x$, we refer to the $\eps$-chain as an {\BF $\eps$-loop}.
    
    We write {\BF $x\tto y$} if, for every $\e>0$, there exists an $\e$-chain from $x$ to $y$ and
    we set 
    $$
    \text{\BF$\cS$}=\{(x,y):\;x\tto y\}\subset X\times X.
    $$
\end{definition}
%

In this article, we will not consider any other type of chains besides Conley chains and shadow chains. 
Hence, we often simply write ``$(\eps,T)$-chain'' to refer to a Conley chain and ``$\eps$-chain'' to refer to a shadow chain.

\medskip
The example below shows that, in case of semiflows of differential equations, $\eps$-chains can be generated as solutions of the same differential equation plus some ``small'' control.
\begin{example}
    \label{ex: DE}
    Let $(X,\|\cdot\|)$ be a Banach space and let $g:X\to X$ be locally Lipschitz.
    Assume that the
    differential equation
    \beqn
    \label{eq: x'}
    \frac{dz}{dt}=g(z),\quad z\in X,
    \eeqn
    gives rise to a (global) semiflow $F$ and that there is some bounded ball $B$ that is forward-invariant under $F$.
    Finally, assume that $g$ is Lipschitz on $B$ with Lipschitz constant $L$, namely
    \beqn
    \label{eq: Lip}
    \|g(x)-g(y)\|\leq L \|x-y\|\text{ for $x,y\in B$,}
    \eeqn
    and that, for some $T\geq1$ and $\eps>0$, a ``control function'' $u:\mathbb R\to X$ satisfies
    $$
    \int_t^{t+1} \|u(s)\|ds<\eps\cdot e^{-L}\text{ for each $t\in[0, T-1]$}.
    $$
    Let $\gamma:[0,T]\to B$ be a solution of the ``perturbed'' differential equation
    $$
    \frac{dz}{dt}=g(z(t))+u(t)
    $$
    and set $x_0=\gamma(0)$ and $y_0=\gamma(T)$.
    Then, by the Gronwall inequality, $\gamma$ is an $\eps$-chain for the semiflow $F$ from $x_0$ to $y_0$.

If $X$ is a Hilbert space with inner product $\langle\cdot,\cdot\rangle$ then, instead of \eqref{eq: Lip}, we can require that
$$
\langle g(x)-g(y), x-y\rangle\leq L \langle x-y,x-y\rangle\text{ for $x,y\in B$}.
$$
\end{example}
%

The next result is of crucial importance for this article.
\begin{proposition}
    \label{prop: closed and transitive}
    The relations $\cC$ and $\cS$ are closed and transitive.
\end{proposition}
Before getting to the proof of the claim,
we go over a technical lemma.
\begin{lemma}[Concatenation of two shadow chains]
\label{lemma: Concatenation}
Let $\varepsilon>0$ and let $y\in X$. Then there exists $\delta>0$ such that, given any $\delta$-chain from $x$ to $y$ and any $\delta$-chain from $y$ to $z$, their concatenation is an $\eps$-chain from $x$ to $z$.
\end{lemma}

\begin{proof}
By continuity of $F$ over the compact set $[0,1]\times\{y\}$, there is an
$\eta>0$ such that
\beqn
\label{eq: unif cont}
        d(u,y)<\eta 
        \quad\text{ implies }\quad
        d(F^\tau(u),F^\tau(y))<\varepsilon/3
        \qquad\text{for all }\tau\in[0,1].
\eeqn

Set $\delta=\min\{\varepsilon/3,\eta\}$, let $\gamma_1:[0,T_1]\to X$ be a $\delta$-chain from $x$ to $y$ and $\gamma_2:[0,T_2]\to X$ be a $\delta$-chain from $y$ to $z$ and let $\gamma:[0,T_1+T_2]\to X$ be the concatenation of $\gamma_1$ and $\gamma_2$.

Let $\tau\in[0,1]$ and $t\in[0,T_1+T_2-\tau]$.
If $T_1\not\in(t,t+\tau)$, the shadow-chain estimate~\eqref{eq: S} for $\gamma$ follows directly from the corresponding estimate for $\gamma_1$ or $\gamma_2$.
Assume now that $T_1\in(t,t+\tau)$, so that $0\in(t-T_1,t-T_1+\tau)\subset(-1,1)$.
Notice first that, since $\gamma_1(T_1)=y$, $\gamma|_{(t,T_1]}=\gamma_1|_{(t,T_1]}$ and $\gamma_1$ is a $\delta$-chain, we have that
\[
        d(y,F^{T_1-t}(\gamma_1(t)))
        =d(\gamma_1(t+(T_1-t)),F^{T_1-t}(\gamma_1(t)))
        <\delta\le\eta
\]
so that, by~\eqref{eq: unif cont},
\[
        d(F^{t-T_1+\tau}(y),F^{\tau}(\gamma_1(t)))
        =
        d(F^{t-T_1+\tau}(y),F^{t-T_1+\tau}(F^{T_1-t}(\gamma_1(t))))<\varepsilon/3.
\]
Moreover, since $\gamma_2$ is a $\delta$-chain and $\gamma(T_1)=\gamma_2(0)=y$,
\[
        d(\gamma_2(t+\tau-T_1),F^{t+\tau-T_1}(y))
        <\delta\le\varepsilon/3.
\]
Therefore, by the triangular inequality,
\[
    d(\gamma(t+\tau),F^\tau(\gamma(t)))
    =
    d(\gamma_2(t+\tau-T_1),F^{\tau}(\gamma_1(t)))
        <2\varepsilon/3<\varepsilon.
\]
Thus, the concatenation is an $\varepsilon$-chain.
\end{proof}

The proofs of the following corollaries are left to the reader.

\begin{corollary}[Finite concatenation]
\label{cor: concatenation}
Let $p_0,\ldots,p_k\in X$ and $\varepsilon>0$. 
There exists a
$\delta>0$ such that, if for each $i=1,\ldots,k$ there is a
$\delta$-chain from $p_{i-1}$ to $p_i$, then the concatenation of all these chains is an $\varepsilon$-chain from $p_0$ to $p_k$.
\end{corollary}
\begin{corollary}[Long loops]
\label{cor: long loops}
If $x\tto  x$, then, for every $\varepsilon>0$ and every $L>0$, there exists an $\varepsilon$-loop through $x$ whose length is larger than $L$.
\end{corollary}
%


\begin{proof}[Proof of Proposition~\ref{prop: closed and transitive}]
\par\noindent
{\BF The relation $\cC$.}
Transitivity follows immediately from the fact that each $(\eps,T)$-chain 
from $x$ to $y$ can be concatenated to any other $(\eps,T)$-chain 
from $y$ to $z$ to create a new $(\eps,T)$-chain 
from $x$ to $z$. 

Closedness is a direct consequence of the continuity of $F$.
Consider first $\oto$ and let \(x_n\to x\) and \(y_n\to y\) and assume \(x_n\oto y_n\) for every \(n\). 
We show that \(x\oto y\).
Fix \(\varepsilon>0\) and \(T>0\).
We will use the following elementary observation: any \((\delta,T)\)-chain can
be refined, without changing its endpoints or increasing its error, so that all its times belong to \([T,2T]\). Indeed, if one leg has time \(s\ge T\), choose an integer \(q\ge 1\) such that \(s/q\in [T,2T]\), and replace that leg by \(q\) legs of equal time \(s/q\), inserting the exact orbit points in between. 
All new intermediate errors are zero, and the final error is the old one.

Since \(F\) is jointly continuous and \([T,2T]\) is compact, we have
\[
\sup_{t\in[T,2T]} d(F^t(x_n),F^t(x))\longrightarrow 0 .
\]
Choose \(n\) so large that
\[
d(y_n,y)<\varepsilon/3
\quad\text{and}\quad
\sup_{t\in[T,2T]} d(F^t(x_n),F^t(x))<\varepsilon/3 .
\]
Since \(x_n\oto y_n\), there is an \((\varepsilon/3,T)\)-chain from
\(x_n\) to \(y_n\). 
By the observation above, we may assume its times all lie in \([T,2T]\). 
Write it as
\[
x_n=c_0,c_1,\ldots,c_m=y_n,
\qquad t_i\in[T,2T],
\]
with
\[
d(F^{t_i}(c_i),c_{i+1})<\varepsilon/3 .
\]
Replace only the first and last points by \(x\) and \(y\), obtaining
\[
x,c_1,\ldots,c_{m-1},y .
\]
If \(m=1\), namely the modified chain has the single leg from \(x\) to \(y\), then
\[
d(F^{t_0}(x),y)
\le d(F^{t_0}(x),F^{t_0}(x_n))
+d(F^{t_0}(x_n),y_n)
+d(y_n,y)
<\varepsilon.
\]
Assume now that $m\geq2$.
Then, for the first leg,
\[
d(F^{t_0}(x),c_1)
\le d(F^{t_0}(x),F^{t_0}(x_n))
   +d(F^{t_0}(x_n),c_1)
<2\varepsilon/3<\varepsilon .
\]
The intermediate legs are unchanged. For the last leg,
\[
d(F^{t_{m-1}}(c_{m-1}),y)
\le d(F^{t_{m-1}}(c_{m-1}),y_n)+d(y_n,y)
<2\varepsilon/3<\varepsilon .
\]


Thus, there is an \((\varepsilon,T)\)-chain from \(x\) to \(y\). 
Since \(\varepsilon\) and \(T\) were arbitrary, \(x\oto y\), namely
\(\cC\) is closed.

\medskip\noindent
{\BF The relation \(\cS\).}
Transitivity is an immediate consequence of Lemma~\ref{lemma: Concatenation}.
Now, 
let \(x_n\to x\), \(y_n\to y\),
and assume \(x_n\tto y_n\) for every \(n\). 
Fix \(\varepsilon>0\). 
Since \(F\) is jointly continuous and \([0,1]\) is compact, the convergence
\(x_n\to x\) implies
\[
\sup_{\tau\in[0,1]} d(F^\tau(x_n),F^\tau(x))\to0 .
\]
Choose \(n\) so large that
\[
    d(y_n,y)<\varepsilon/3
    \quad\text{and}\quad
    \sup_{\tau\in[0,1]} d(F^\tau(x_n),F^\tau(x))<\varepsilon/3 .
\]
Since \(x_n\tto y_n\), there exists an \(\varepsilon/3\)-chain \(\gamma_n:[0,T]\to X\) from \(x_n\) to \(y_n\). Define
\[
\gamma(t)=
\begin{cases}
x, & t=0,\\
\gamma_n(t), & 0<t<T,\\
y, & t=T .
\end{cases}
\]
Then \(\gamma(0)=x\) and \(\gamma(T)=y\). We verify that \(\gamma\) is an
\(\varepsilon\)-chain.

If \(0<t<t+\tau<T\), the estimate follows from the fact that \(\gamma_n\) is
an \(\varepsilon/3\)-chain. If \(t=0\) and \(t+\tau<T\), then
\[
d(\gamma(\tau),F^\tau(x))
\le d(\gamma_n(\tau),F^\tau(x_n))
   +d(F^\tau(x_n),F^\tau(x))
<2\varepsilon/3 .
\]
If \(t+\tau=T\), then
\[
d(\gamma(T),F^\tau(\gamma(t)))
\le d(y,y_n)+d(y_n,F^\tau(\gamma_n(t)))
<2\varepsilon/3 ,
\]
with the same estimate also covering the case \(t=0\), \(\tau=T=1\). Hence
\[
d(\gamma(t+\tau),F^\tau(\gamma(t)))<\varepsilon
\]
for every \(\tau\in[0,1]\) and every \(t\in[0,T-\tau]\). Thus \(\gamma\) is an
\(\varepsilon\)-chain from \(x\) to \(y\). 
Since we can do this for every \(\varepsilon>0\), it follows that
\(x\tto y\). 
Hence, \(\cS\) is closed.
\end{proof}




\medskip\noindent
{\bf Nodes and graph of a closed and transitive relation.}
\medskip
The fact that $\cC$ and $\cS$ are not the same relation is not critical to us.
What we do care about is 
whether $\cC$ and $\cS$ share the same qualitative features 
in the sense illustrated below.
%
\begin{definition}
    \label{def: nodes and graph}
    Let $\ato$ be a closed and transitive binary relation on $X$ and $x,y\in X$. 
    Denote by $R\subset X\times X$ the set of all pairs $(x,y)$ such that $x\ato y$. 
    We say that {\BF $x$ is $R$-equivalent to $y$} if $x\ato y$ and $y\ato x$.
    We say that {\BF $x$ is $R$-recurrent} if there is a $t>0$ such that $x$ is $R$-equivalent to $F^t(x)$.
    If all points in $M\subset X$ are mutually $R$-equivalent, we say that {\BF $M$ is $R$-equivalent}.
    We denote the set of all $R$-recurrent points by {\BF $\cR_{R}$}.
    We call {\bf node} of $\cR_R$ each maximal $R$-equivalent set.
    Given two nodes $M,N$, we write {\BF $M\ato N$} if there are $x\in M$ and $y\in N$ such that $x\ato y$.
    We denote by {\BF $\Gamma_R$} the graph whose nodes are the nodes of $\cR_{R}$ and such that there is an edge from node $M$ to node $N$ if $M\ato N$.
\end{definition}
The proposition below grants that the definition of edge given above is well-posed.
%
\begin{proposition}
    Let $R$ be a closed and transitive relation and write $x\ato y$ if $(x,y)\in R$.
    Then:
    \begin{enumerate}
        \item each node of $R$ is closed;
        \item given any two nodes $M,N$ of $R$, if $x\ato y$ for an $x\in M$ and an $y\in N$, then $x'\ato y'$ for every $x'\in M, y'\in N$.
    \end{enumerate}
\end{proposition}
\begin{proof}
  {\bf (1)} Let $N$ be a node of $R$, fix \(x\in N\) and set  
  \[
    M=\{y\in X: x\ato y \text{ and } y\ato x\}.
  \]
  We claim that $M=N$.
  The inclusion \(N\subset M\) is immediate, since all points of \(N\) are mutually \(R\)-equivalent.
  Conversely, let \(y\in M\)
  and \(z\in N\). 
 Since \(z\) and \(x\)
    belong to \(N\), we have
\[
z\ato x
\qquad\text{and}\qquad
x\ato z.
\]
Using transitivity together with \(x\ato y\) and \(y\ato x\), we get
\[
z\ato x,\ x\ato y \quad\Longrightarrow\quad z\ato y,
\]
and
\[
y\ato x,\ x\ato z \quad\Longrightarrow\quad y\ato z.
\]
Thus, \(y\) is mutually \(R\)-equivalent to every point of \(N\). 
By maximality of \(N\), it follows that \(y\in N\) and therefore $M\subset N$, namely
$$
M=N.
$$

Now
\[
\{y\in X:x\ato y\}=\{y\in X:(x,y)\in R\}
\]
is closed, because it is the preimage of the closed set \(R\subset X\times X\)
under the continuous map
\[
X\to X\times X,\qquad y\mapsto (x,y).
\]
Similarly,
\[
\{y\in X:y\ato x\}=\{y\in X:(y,x)\in R\}
\]
is closed. 
Therefore,
\[
N=\{y\in X:x\ato y\}\cap \{y\in X:y\ato x\}
\]
is closed.

{\bf (2)}
Since $x,x'\in M$, we have that $x'\ato x$.
Since $y,y'\in N$, we have that $y\ato y'$.
Hence,
$$
x'\ato x\ato y\ato y'
$$
and so
$$
x'\ato y'.
$$
\end{proof}

\medskip\noindent
{\bf Chain relations and chain streams.}
To our knowledge, the particular role played in the qualitative study of dynamical systems by closed and transitive binary relations that extend the relation ``being on the orbit of'' appeared for the first time in the article by Joseph Auslander~\cite{Aus64} where he introduced generalized recurrence.
In~\cite{DY25}, we started a systematic study of this kind of relations by introducing the concept of stream as follows and studying some of its fundamental properties.
\begin{definition}[De Leo, Yorke, 2025~\cite{DY25}]
    Let $F$ be a semiflow. 
    We denote by 
    $$
    \text{\BF$\cO_F$} = \{(x,y)| x\in X,\,y=F^t(x)\text{ for some }t\geq0\}\subset X\times X
    $$
    the orbit relation of $F$.
    We say that a closed and transitive binary relation $S\subset X\times X$ is an {\BF $F$-stream} (or simply {\bf stream} when there is no ambiguity) if $\cO_F\subset S$.
\end{definition}
%
\medskip\noindent
{\bf The Conley and shadow chain relations defined above are not streams.} Indeed, although they are closed and transitive, neither of them contains necessarily $\cO_F$.

To see this consider, for instance, the semiflow $F$ of the ODE $\dot x=1-x^2$ on $X=[0,1]$.
Then the forward orbit of $0$ under $F$ is the interval $[0,1)$ but $0\oto x$ only for $x=1$ and $0\tto x$ only for $x\in[F^1(0),1]$.

Indeed, in the first case, assume \(x<1\), let
\(
\varepsilon\in(0,(1-x)/2)
\)
and choose \(T>0\) so large that
\[
F^T(0)>1-\varepsilon.
\]
Suppose, by contradiction, that there is an \((\varepsilon,T)\)-chain
from \(0\) to \(x\):
\[
0=x_0,x_1,\ldots,x_N=x,
\qquad t_i\ge T.
\]
Since the semiflow is order-preserving and \(x_i\in[0,1]\), we have
\[
F^{t_i}(x_i)\ge F^T(0)>1-\varepsilon
\]
for every \(i\). 
Hence,
\[
x_{i+1}>1-2\varepsilon
\]
for every \(i\), because
\[
|x_{i+1}-F^{t_i}(x_i)|<\varepsilon.
\]
In particular, for the final point,
\[
x=x_N>1-2\varepsilon,
\]
which contradicts our choice of $\eps$,
so
\[
0\not\succeq_C x
\]
for every \(x<1\).
The reader can verify that, since $F^t(0)\to1$ for $t\to\infty$, $0\oto1$. 

In the second case, let \(\varepsilon>0\) and let
\[
\gamma:[0,T]\to[0,1]
\]
be an \(\varepsilon\)-chain from \(0\) to \(x\). 
Since \(T\ge1\), we may
apply the shadow estimate~\eqref{eq: S} with initial time \(T-1\) and \(\tau=1\). Thus
\[
d\bigl(x,F^1(\gamma(T-1))\bigr)
=
d\bigl(\gamma(1+T-1),F^1(\gamma(T-1))\bigr)
<\varepsilon.
\]
Since \(\gamma(T-1)\in[0,1]\) and the flow is order-preserving, we have that
\[
F^1(\gamma(T-1))\ge F^1(0),
\]
so that
\[
x>F^1(0)-\varepsilon.
\]
Since \(\varepsilon>0\) is arbitrary, we get
\[
x\ge F^1(0).
\]
Finally, the orbit $s\to F^s(0)$ itself, restricted to $[0,t]$, is an $\eps$-chain from $0$ to $F^t(0)$ for every $t\geq1$ for every $\eps>0$ and so $0\tto F^t(0)$ for every $t\geq1$. 
Since $\cS$ is closed, this also shows that $0\tto1$.


\medskip
Although $\cC$ and $\cS$ are not streams, we prove in Proposition~\ref{prop: C and O U C pari son} below that each of them 
gives rise to the same chain-recurrent set, the same nodes and the same graph.
In particular, this justifies and clarifies the definition of chain stream for continuous-time semiflows given in~\cite{DY25,DLY26,DLY26b}.

\medskip
In order to prove Proposition~\ref{prop: C and O U C pari son} we need the following general result and a few preparatory lemmas.
\begin{proposition}
    \label{prop: chain rel vs chain stream}
    Let \(F\) be a continuous-time semiflow on \(X\).
    Let \(R\subset X\times X\) be a closed and transitive relation, and set
    \[
    S_R=R\cup\cO_F.
    \]
    Assume that \(R\) satisfies the following 
    properties:
    \begin{enumerate}
    \item[(OC0)] $S_R$ is closed and transitive (and so is a stream);
    \item[(OC1)] if \(x\in \cR_R\), then the forward orbit of $x$ is an $R$-equivalent set;
    \item[(OC2)] if, for some \(t>0\),
    \(
    (F^t(x),x)\in S_R,  
    \)
    then \(x\in \cR_R\).
\end{enumerate}
    Then \(R\) and \(S_R\) have the same recurrent set, the same nodes, and the same graph.
\end{proposition}
\begin{proof}
    Since \(R\subset S_R\), we immediately have
    \[
    \cR_R\subset \cR_{S_R}.
    \]
    Conversely, let \(x\in \cR_{S_R}\). 
    Then there is a \(t>0\) such that \(x\) and \(F^t(x)\) are \(S_R\)-equivalent. 
    In particular, $(F^t(x),x)\in S_R$.
    By (OC2), this implies \(x\in \cR_R\). 
    Hence
    \[
    \cR_{S_R}=\cR_R.
    \]

    We now compare the nodes. 
    Let \(x,y\in \cR_R=\cR_{S_R}\). 
    We claim that $x$ is $S_R$-equivalent to $y$ if and only if $x$ is $R$-equivalent to $y$.
    We prove the non-trivial implication.
    Assume that $x$ is $S_R$-equivalent to $y$.
    The only non-trivial case is when $(x,y)\in\cO_F$ but in this case, by (OC1), we know that the orbit of $x$ is $R$-equivalent and so $x$ and $y$ are $R$-equivalent as well.
    Hence, the nodes of $R$ and $S_R$ coincide.

    Finally, we compare the graphs. 
    Since \(R\subset S_R\), every \(R\)-edge is an \(S_R\)-edge. 
    Conversely, suppose there is an \(S_R\)-edge from a node \(M\) to a node \(N\). 
    Then there exist \(x\in M\) and \(y\in N\) such that
\[
(x,y)\in S_R.
\]
If \((x,y)\in R\), this is already an \(R\)-edge. 
If instead \((x,y)\in\cO_F\), then
\(y=F^t(x)\) for some \(t\ge0\). 
Since \(x\in M\subset \cR_R\), (OC1) gives
\[
(x,y)\in R.
\]
Thus every \(S_R\)-edge is already an \(R\)-edge. Hence the two 
graphs coincide.
\end{proof}
We prove below that $\cC$ and $\cS$ satisfy the three properties above.
\begin{lemma}[OC0]
\label{lemma: OC0}
Let $R=\cC,\cS$.
Then
\(
S_R=\cO_F\cup R
\)
is an \(F\)-stream. 
\end{lemma}

\begin{proof}
We need to prove that $S_R$ is closed and transitive.

Since $R$ is closed, in order to show that $S_R$ is closed it is enough to show that $$\overline{\cO_F}\cup R=\cO_F\cup R.$$
In case $R=\cS$, this is due to the fact that, for every $x\in X$,  $(x,F^t(x))\in\cS$ for $t\geq1$.
Now, set
$$\cO_F^{[0,1]}=\{(x,F^t(x))|x\in X,t\in[0,1]\},\quad
\cO_F^{[1,\infty)}=\{(x,F^t(x))|x\in X, t\geq1\}.$$
Then $\cO_F^{[0,1]}$ is closed and
$$
\cO^{[1,\infty)}_F\subset\cS
$$
so that, since $\cS$ is closed,
$$
\overline{\cO_F}\cup\cS=\overline{\cO_F^{[0,1]}}\cup\overline{\cO_F^{[1,\infty)}}\cup\cS=\cO_F^{[0,1]}\cup\cS=\cO_F\cup\cS.
$$

Consider now the case $R=\cC$.
Since \(\cC\) is closed, it is enough to consider limits of pairs in \(\cO_F\). 
Let
\[
x_n\to x,\qquad F^{t_n}(x_n)\to y.
\]
If \((t_n)\) is bounded, then, after passing to a subsequence, \(t_n\to t\ge0\).
By continuity of \(F\),
\[
y=\lim_{n\to\infty}F^{t_n}(x_n)=F^t(x),
\]
and so \((x,y)\in \cO_F\).

Assume instead that \(t_n\to+\infty\) and 
fix an \(\varepsilon>0\) and a \(T>0\). 
Choose \(n\) so large that
\[
t_n\ge 2T,
\qquad
d(F^T(x_n),F^T(x))<\varepsilon,
\qquad
d(F^{t_n}(x_n),y)<\varepsilon.
\]
Then
\[
x,\ F^T(x_n),\ y
\]
with transition times
\[
T,\qquad t_n-T
\]
is an \((\varepsilon,T)\)-chain from \(x\) to \(y\). Indeed,
\[
d(F^T(x),F^T(x_n))<\varepsilon
\]
and
\[
d\bigl(F^{t_n-T}(F^T(x_n)),y\bigr)
=
d(F^{t_n}(x_n),y)<\varepsilon.
\]
Since \(\varepsilon>0\) and \(T>0\) were arbitrary, \(x\oto y\), so that 
\((x,y)\in \cC\).
Therefore, every limit point of \(\cO_F\) lies in \(O_F\cup \cC\) and so \(\cO_F\cup \cC\) is closed.

\medskip
It remains to prove transitivity. 
Since \(R\) is transitive and \(\cO_F\) is transitive, only the mixed cases need to be considered.

First assume that
\[
(x,y)\in R
\qquad\text{and}\qquad
(y,z)\in\cO_F.
\]
Then \(z=F^t(y)\) for some \(t\ge0\). 
For \(R=\cC\), the implication
\[
x\oto y \Longrightarrow x\oto F^t(y)
\]
is obtained by adding \(t\) to the final transition time of a sufficiently accurate Conley chain from \(x\) to \(y\). 
Hence \(x\oto z\). 
For \(R=\cS\), append the exact orbit segment from \(y\) to \(F^t(y)\).
The resulting curve is admissible because the original shadow chain piece has length at least \(1\). 
The crossing estimate is the same one-jump estimate used in Lemma~\ref{lemma: Concatenation}.
Hence, \(x\tto z\).

Now, assume that
\[
(x,y)\in\cO_F
\qquad\text{and}\qquad
(y,z)\in R.
\]
Then \(y=F^t(x)\) for some \(t\ge0\). 
For \(R=\cC\), the implication
\[
F^t(x)\oto z \Longrightarrow x\oto z
\]
is obtained by adding \(t\) to the first transition time of a Conley chain from
\(F^t(x)\) to \(z\). 
For \(R=\cS\), prepend the exact orbit segment from \(x\) to \(F^t(x)\) to a sufficiently accurate shadow chain from \(F^t(x)\) to
\(z\).
The resulting curve is admissible because the shadow-chain piece has length at least \(1\). 
The crossing estimate is the same one-jump estimate used in Lemma~\ref{lemma: Concatenation}. 
Hence \((x,z)\in R\) in either case and so $S_R$ is a $F$-stream.
%
\end{proof}
\begin{lemma}
    \label{lemma: eq 1}
    If \(u\oto v\) then, for every \(r\ge0\),
    $$
    F^r(u)\oto v
    \qquad\text{and}\qquad
    u\oto F^r(v).
    $$
\end{lemma}
\begin{proof}
    Fix \(\varepsilon>0\) and \(T>0\). 
    To prove the first implication, take an \((\varepsilon,T+r)\)-chain from \(u\) to \(v\),
\[
u=x_0,x_1,\ldots,x_n=v,
\qquad t_i\ge T+r.
\]
    Then
\[
F^r(u),x_1,\ldots,x_n=v
\]
    with transition times
\[
t_0-r,t_1,\ldots,t_{n-1}
\]
    is an \((\varepsilon,T)\)-chain from \(F^r(u)\) to \(v\).

    To prove the second implication, choose \(\delta\in(0,\varepsilon)\) such that
\[
d(w,v)<\delta
\quad\Longrightarrow\quad
d(F^r(w),F^r(v))<\varepsilon .
\]
    Take a \((\delta,T)\)-chain from \(u\) to \(v\),
\[
u=x_0,x_1,\ldots,x_n=v,
\qquad t_i\ge T.
\]
    Then
\[
u=x_0,x_1,\ldots,x_{n-1},F^r(v)
\]
    with transition times
\[
t_0,\ldots,t_{n-2},t_{n-1}+r
\]
    is an \((\varepsilon,T)\)-chain from \(u\) to \(F^r(v)\), which proves the claim.
\end{proof}

\begin{lemma}
\label{lemma: x >= x} 
For the Conley and shadow chain relations,
\[
x\in \cR_\cC \iff x\oto x,
\qquad
x\in \cR_\cS \iff x\tto x.
\]
\end{lemma}

\begin{proof}
Consider first the Conley relation.
If \(x\in \cR_\cC\), then, by definition, there exists \(t>0\) such that \(x\) is Conley chain-equivalent to \(F^t(x)\).
Thus,
\[
x\succeq_C F^t(x)
\qquad\text{and}\qquad
F^t(x)\succeq_C x.
\]
By transitivity,
\[
x\succeq_C x.
\]

Conversely, assume \(x\oto x\). 
By Lemma~\ref{lemma: eq 1}, applied with \(u=v=x\), we get
\[
x\oto F^1(x)
\qquad\text{and}\qquad
F^1(x)\oto x.
\]
Therefore, \(x\) is Conley chain-equivalent to \(F^1(x)\), and so
\[
x\in \cR_\cC.
\]

Consider now the shadow relation. 
If \(x\in \cR_\cS\), then, by definition, there exists \(t>0\) such that \(x\) is shadow chain-equivalent to \(F^t(x)\). 
Hence, by transitivity,
\[
x\tto x.
\]

Conversely, assume \(x\tto x\). 
By Corollary~\ref{cor: long loops}, for every \(\varepsilon>0\) there are sufficiently accurate shadow loops through \(x\) with arbitrarily large length. 
We show that
\[
F^1(x)\tto x.
\]
Let \(\varepsilon>0\). 
Choose a sufficiently accurate shadow loop
\[
\gamma:[0,L]\to X,\qquad \gamma(0)=\gamma(L)=x,
\]
with \(L>2\). 
By finite-time tracking on \([0,1]\), choosing the loop accuracy small enough gives
\[
d(\gamma(1),F^1(x))
\]
so small that
\[
d(F^\tau(\gamma(1)),F^\tau(F^1(x)))<\varepsilon/2
\qquad\text{for all }\tau\in[0,1].
\]
Define
\[
\widetilde\gamma:[0,L-1]\to X
\]
by
\[
\widetilde\gamma(0)=F^1(x),
\qquad
\widetilde\gamma(t)=\gamma(1+t)\quad\text{for }t>0.
\]
Since \(L-1>1\), this is an admissible curve from \(F^1(x)\) to \(x\). 
For initial times \(t>0\), the shadow estimates are inherited from \(\gamma\). 
For \(t=0\) and \(\tau\in[0,1]\),
\[
d(\widetilde\gamma(\tau),F^\tau(\widetilde\gamma(0)))
=
d(\gamma(1+\tau),F^\tau(F^1(x)))
\]
and this is bounded by
\[
d(\gamma(1+\tau),F^\tau(\gamma(1)))
+
d(F^\tau(\gamma(1)),F^\tau(F^1(x)))
<\varepsilon.
\]
Thus,
\[
F^1(x)\tto x.
\]
On the other hand, the exact orbit segment
\[
r\mapsto F^r(x),\qquad 0\le r\le1,
\]
is an \(\eps\)-chain from \(x\) to \(F^1(x)\) for every $\eps>0$. 
Hence,
\[
x\tto F^1(x).
\]
Therefore \(x\) is shadow chain-equivalent to \(F^1(x)\), and so
\[
x\in \cR_\cS.
\]
\end{proof}

\begin{lemma}[OC1]
\label{lemma: OC1}
Let $R=\cC,\cS$.
Then $x\in\cR_R$ if and only if $\{F^t(x):t\ge0\}$ is an $R$-equivalent set.
\end{lemma}

\begin{proof}
{\BF$R=\cC$.}
Assume  that \(x\in \cR_\cC\). 
By Lemma~\ref{lemma: x >= x}, 
$
x\oto x
$
so that, applying Lemma~\ref{lemma: eq 1} with \(u=v=x\), we get, for every \(t\ge0\),
\[
x\oto F^t(x)
\qquad\text{and}\qquad
F^t(x)\oto x.
\]
Thus, every point of the forward orbit of \(x\) is Conley chain-equivalent to \(x\). 
By transitivity, any two points of the forward orbit are Conley chain-equivalent. 
Hence,
\[
\{F^t(x):t\ge0\}
\]
is a Conley chain-equivalent set.

Conversely, assume that the forward orbit of \(x\) is a Conley
chain-equivalent set. Since \(F^1(x)\) belongs to this orbit, \(x\) is Conley
chain-equivalent to \(F^1(x)\). Therefore, by definition, \(x\in \cR_\cC\).

\medskip\noindent
{\BF$R=\cS$.}
If $\{F^t(x):t\ge0\}$ is shadow chain-equivalent, then clearly $x\in\cR_\cS$.
Assume now that $x\in\cR_\cS$.
We need to show that $\{F^t(x):t\ge0\}$ is shadow chain-equivalent. 
By transitivity, it is enough to show that
\[
        x\tto F^s(x)
        \qquad\text{and}\qquad
        F^s(x)\tto x
        \qquad\text{for every }s\ge0.
\]

Fix $s\ge0$ and $\varepsilon>0$. Since, by transitivity, $x\tto x$, by Corollary~\ref{cor: long loops} there are
arbitrarily accurate shadow loops based at $x$ with arbitrarily large
lengths.

To prove $x\tto  F^s(x)$, take a sufficiently accurate shadow loop
\[
        \gamma:[0,L]\to X,\qquad \gamma(0)=\gamma(L)=x,
\]
and append to it the exact orbit segment from $x$ to $F^s(x)$. 
Thus, define
\[
        \Gamma:[0,L+s]\to X
\]
by
\[
        \Gamma(t)=\gamma(t)\quad 0\le t\le L,
        \qquad
        \Gamma(L+r)=F^r(x)\quad 0\le r\le s.
\]
If the shadow error of $\gamma$ is sufficiently small, then $\Gamma$ is an $\varepsilon$-shadow chain from $x$ to $F^s(x)$. 
The only point to check is when the interval $[t,t+\tau]$ crosses the
joining time. 
This can be done using the same one-jump estimate argument used in the proof of Lemma~\ref{lemma: Concatenation}. 
Therefore, after choosing the error of the first piece small enough, the concatenated curve is a shadow $\varepsilon$-chain.
Therefore $x\tto  F^s(x)$.

To prove $F^s(x)\tto  x$, take a sufficiently accurate shadow loop
\[
        \gamma:[0,L]\to X,\qquad \gamma(0)=\gamma(L)=x,
\]
with $L>s+1$.
This shadow loop exists by Corollary~\ref{cor: long loops}. 
By finite-time tracking on
$[0,s]$, if the shadow error of $\gamma$ is sufficiently small, then
\[
        d(\gamma(s),F^s(x))
\]
is so small that
\[
        d(F^\tau(\gamma(s)),F^\tau(F^s(x)))<\varepsilon/2
        \qquad\text{for all }\tau\in[0,1].
\]
Define
\[
        \widetilde\gamma:[0,L-s]\to X
\]
by
\[
        \widetilde\gamma(0)=F^s(x),
        \qquad
        \widetilde\gamma(t)=\gamma(s+t)\quad\text{for }t>0.
\]
Since $L-s>1$, this is an admissible curve. For initial times $t>0$, the
shadow-chain estimates are inherited from $\gamma$. For $t=0$ and
$\tau\in[0,1]$,
\[
\begin{aligned}
d(\widetilde\gamma(\tau),F^\tau(\widetilde\gamma(0)))
&=d(\gamma(s+\tau),F^\tau(F^s(x)))\\
&\le d(\gamma(s+\tau),F^\tau(\gamma(s)))
   +d(F^\tau(\gamma(s)),F^\tau(F^s(x)))\\
&<\varepsilon .
\end{aligned}
\]
Thus, $F^s(x)\tto x$.
\end{proof}
\begin{lemma}[OC2]
    \label{lemma: OC2}
    Let $R=\cC,\cS$.
    Then 
    \[
    (F^t(x),x)\in \cO_F\cup R \text{ for some }t>0
    \quad\Longrightarrow\quad
    x\in \cR_R.
    \]
\end{lemma}

\begin{proof}
{\BF $R=\cC$.} Assume that, for some \(t>0\),
\[
(F^t(x),x)\in \cO_F\cup\cC.
\]
There are two cases.
First, suppose that
\[
F^t(x)\oto x.
\]
We show that \(x\oto x\). 
Let \(\varepsilon>0\) and \(T>0\). 
Since \(F^t(x)\oto x\), there exists an \((\varepsilon,T+t)\)-chain
\[
F^t(x)=x_0,x_1,\ldots,x_n=x
\]
with transition times \(s_i\ge T+t\). 
Then
\[
x,x_1,\ldots,x_n=x
\]
with transition times
\[
s_0+t,s_1,\ldots,s_{n-1}
\]
is an \((\varepsilon,T)\)-chain from \(x\) to itself, since
\[
F^{s_0+t}(x)=F^{s_0}(F^t(x)).
\]
Hence, \(x\oto x\) and so, by Lemma~\ref{lemma: x >= x}, \(x\in\cR_\cC\).

Second, suppose that
\[
(F^t(x),x)\in \cO_F.
\]
Then there exists \(s\ge0\) such that
\[
F^s(F^t(x))=x.
\]
Thus,
\[
F^{t+s}(x)=x,
\]
with \(t+s>0\). 
For every \(\varepsilon>0\) and \(T>0\), choose an integer \(m\ge1\) such that
\[
m(t+s)\ge T.
\]
Then the one-leg exact orbit segment from \(x\) to itself with transition time
\(m(t+s)\) is an \((\varepsilon,T)\)-chain from \(x\) to itself. 
Hence
\(x\oto x\) and so, by Lemma~\ref{lemma: x >= x}, \(x\in \cR_\cC\).

\medskip\noindent
{\BF$R=\cS$.}
Assume that, for some \(t>0\),
\[
(F^t(x),x)\in \cO_F\cup\cS.
\]
As above, there are two cases.

First, suppose that
\[
F^t(x)\tto x.
\]
We show that \(x\tto x\). 
Fix \(\varepsilon>0\). Choose a sufficiently
accurate shadow chain
\[
\gamma:[0,L]\to X
\]
from \(F^t(x)\) to \(x\). 
Define a curve
\[
\Gamma:[0,t+L]\to X
\]
by
\[
\Gamma(r)=F^r(x)\quad 0\le r\le t,
\]
and
\[
\Gamma(t+r)=\gamma(r)\quad 0\le r\le L.
\]
If the shadow error of \(\gamma\) is sufficiently small, then \(\Gamma\) is an \(\varepsilon\)-loop through \(x\). 
Indeed, away from the joining time \(t\), the estimate is either exact or inherited from \(\gamma\). 
If an interval crosses \(t\), write it as \([t-a,t+b]\), with \(a,b\ge0\) and \(a+b\le1\). 
Then
\[
F^{a+b}(\Gamma(t-a))
=
F^{a+b}(F^{t-a}(x))
=
F^b(F^t(x)),
\]
while
\[
\Gamma(t+b)=\gamma(b).
\]
Thus, the crossing estimate follows from the shadow-chain estimate for \(\gamma\), since \(\gamma(0)=F^t(x)\). 
Hence \(x\tto x\) and so, by Lemma~\ref{lemma: x >= x}, \(x\in \cR_\cS\).

Second, suppose that
\[
(F^t(x),x)\in \cO_F.
\]
Then, as before, there exists \(s\ge0\) such that
\[
F^{t+s}(x)=x,
\]
with \(t+s>0\). 
Choose an integer \(m\ge1\) such that
\[
m(t+s)\ge1.
\]
The exact periodic orbit segment
\[
r\mapsto F^r(x),\qquad 0\le r\le m(t+s),
\]
is a \(\eps\)-loop through \(x\) for every $\eps>0$. 
Hence \(x\tto x\) and so, by Lemma~\ref{lemma: x >= x}, \(x\in \cR_\cS\). 
%
\end{proof}
\begin{proposition}
    \label{prop: C and O U C pari son}
    Let $R=\cC,\cS$.
    Then $S_R=\cO_F\cup R$ is a stream and $R$ and $S_R$ have the same chain-recurrent set, the same nodes and the same graph.
\end{proposition}
\begin{proof}
    By Lemmas~\ref{lemma: OC0},~\ref{lemma: OC1} and~\ref{lemma: OC2}, both $\cC$ and $\cS$ satisfy properties (OC0), (OC1), (OC2).
    Hence, the result follows by Proposition~\ref{prop: chain rel vs chain stream}.
\end{proof}

\medskip\noindent
{\BF The binary relations $\cC$ and $\cS$ are not identical.}
This fact is already evident from the example given above of the semiflow of $\dot x=1-x^2$ on $[0,1]$.
Next example shows that the difference between $\cC$ and $\cS$ extends beyond the orbit: there are semiflows for which there are points $x,y$ such that $y$ is not on the orbit of $x$, $x\oto y$ and $x\not\tto y$. 

\begin{example}
\label{ex: C neq S}
Let
\[
A=\{(s,0):s\in\mathbb R\},\qquad
B=\{(s,e^{-s}):s\in\mathbb R\},
\]
and let
\[
X=A\cup B\subset\mathbb R^2
\]
with the metric induced by the Euclidean distance. Define a flow \(F\) on \(X\)
by
\[
F^t(s,0)=(s+t,0),\qquad t\in\mathbb R,
\]
and
\[
F^t(s,e^{-s})=(s-t,e^{-(s-t)}),\qquad t\in\mathbb R.
\]
Thus points on \(A\) move to the right, while points on \(B\) move to the
left.

Notice that $F$ has no compact dynamics. 
Indeed, the positive orbit of
\[
O=(0,0)\in A
\]
is
\[
F^t(O)=(t,0),\qquad t\geq 0,
\]
and escapes every compact subset of \(X\), namely there is no compact subset of \(X\) that attracts the compact set \(\{O\}\).

Let
\[
y=(0,1)\in B.
\]
We show that
\[
O\oto  y
\qquad\text{but}\qquad
O\not\tto  y.
\]

Indeed, let \(\varepsilon>0\) and \(T>0\), choose \(R\geq T\) so large that
$
e^{-R}<\varepsilon
$
and set
\[
c_0=O,\qquad c_1=(R,e^{-R})\in B,\qquad c_2=y.
\]
Then
\[
F^R(c_0)=(R,0),
\]
and therefore
\[
d(F^R(c_0),c_1)=e^{-R}<\varepsilon.
\]
Moreover,
\[
F^R(c_1)=F^R(R,e^{-R})=(0,1)=y=c_2.
\]
Thus \(c_0,c_1,c_2\), with transition times both equal to \(R\), is a
Conley \((\varepsilon,T)\)-chain from \(O\) to \(y\). 
Since \(\varepsilon>0\)
and \(T>0\) were arbitrary, we have
\[
O\oto  y.
\]

We now prove that \(O\not\tto  y\). 
Write
\[
p_A(s)=(s,0),\qquad p_B(s)=(s,e^{-s}),
\]
and define
\[
\pi(p_A(s))=\pi(p_B(s))=s.
\]
Assume, by contradiction, that for every \(\varepsilon>0\) there is an \(\varepsilon\)-chain
\[
\gamma:[0,L]\to X
\]
from \(O\) to \(y\). Choose \(\varepsilon>0\) so small that
\[
-\log\varepsilon>3\varepsilon.
\]

Since \(\gamma(0)=O\in A\) and \(\gamma(L)=y\in B\), the curve \(\gamma\)
must jump at least once from \(A\) to \(B\). Let \(\theta\) be the last time at
which such a jump from \(A\) to \(B\) occurs so that, after time \(\theta\), the
curve remains in \(B\) until the final time \(L\).

Let
\[
a=\lim_{t\to\theta^-}\pi(\gamma(t)),
\qquad
b=\lim_{t\to\theta^+}\pi(\gamma(t)).
\]
Taking \(t\to\theta^-\) and \(s\to\theta^+\) in the shadow estimate
\[
d(\gamma(s),F^{s-t}(\gamma(t)))<\varepsilon,
\qquad 0<s-t\leq 1,
\]
we obtain
\[
d(p_B(b),p_A(a))\leq \varepsilon.
\]
In particular,
\[
|a-b|\le
\varepsilon
\qquad\text{and}\qquad
e^{-b}\leq\varepsilon.
\]
Hence
\[
b\geq -\log\varepsilon.
\]

Now take \(u\in(0,\min\{1,L-\theta\})\). Applying the shadow estimate from a
point just before \(\theta\) to the point \(\gamma(\theta+u)\in B\), and then
letting the initial time tend to \(\theta^-\), gives
\[
\left|\pi(\gamma(\theta+u))-(a+u)\right|\leq\varepsilon.
\]
On the other hand, applying the shadow estimate inside \(B\), from a point just
after \(\theta\) to the same point \(\gamma(\theta+u)\), and then letting the
initial time tend to \(\theta^+\), gives
\[
\left|\pi(\gamma(\theta+u))-(b-u)\right|\leq\varepsilon.
\]
Therefore
\[
|b-a-2u|\leq 2\varepsilon.
\]
Since \(|b-a|\leq\varepsilon\), we get
\[
2u\leq 3\varepsilon.
\]
This holds for every \(u\in(0,\min\{1,L-\theta\})\). Hence
\[
L-\theta\leq \frac{3}{2}\varepsilon.
\]

Finally, applying the shadow estimate inside \(B\) from just after \(\theta\)
to the final time \(L\), and using \(\gamma(L)=y=p_B(0)\), gives
\[
|0-(b-(L-\theta))|\leq\varepsilon.
\]
Thus
\[
b\leq L-\theta+\varepsilon
\leq \frac{5}{2}\varepsilon.
\]
This contradicts
\[
b\geq -\log\varepsilon>3\varepsilon.
\]
Therefore, for all sufficiently small \(\varepsilon>0\), there is no
\(\varepsilon\)-chain from \(O\) to \(y\), namely
\[
O\not\tto  y.
\]
\end{example}
%
%

\noindent
{\BF $\cC$ and $\cS$ for semiflows with ``compact dynamics''.}
As we argued in~\cite{DY25,DLY26,DLY26b,ADY25}, semiflows coming from applications have often some kind of ``compact dynamics'' (see Definition~\ref{def: compact dynamics} below), a fact that grants the fundamental facts that every point has a non-empty limit set and that the set of chain-recurrent points is non-empty. 
It is a natural question whether $\cC$ and $\cS$ lead to the same qualitative description of the dynamics of this kind of semiflows, namely whether they give rise to the same chain-recurrent sets, the same nodes and the same graph.
{\em
The main goal of the present article is proving that this fact does indeed hold true.}

\medskip\noindent
{\bf Semiflows with compact dynamics.}
In~\cite{DLY26}, we introduced the concept of semiflow with compact dynamics as follows.
\begin{definition}[Compact dynamics]
    \label{def: compact dynamics}
    Given an $\eps>0$ and a set $G\subset X$, we set
    $$
    \text{\BF$N_\eps(G)$}=\{y : d(y,G)<\eps\}.
    $$
    Given a semiflow $F$ on $X$, we say that a set $G$ {\bf attracts} a set $K$ under $F$ if, for every $\eps>0$, there exists $T>0$ such that $F^t(K)\subset N_\eps(G)$ for all $t\geq T$.
    The {\bf global attractor} $\cG\subset X$ of a semiflow $F$, when it exists, is a maximal invariant compact set of $X$ that attracts each compact set $K\subset X$.
    We say that $F$ has {\bf compact dynamics} if $F$ has a global attractor.
\end{definition}
When $X$ is locally compact, the global attractor of a semiflow on $X$ with compact dynamics attracts some of its neighborhoods and is uniformly continuous over any compact set $[a,b]\times K\subset[0,\infty)\times X$. 
This facts are used to prove our main results but do not hold necessarily when $X$ is not locally compact, like in case of semiflows generated by partial differential equations.
In order to cover also these important cases, we introduced in~\cite{DLY26b} the concept of strong compact dynamics.
%
\begin{definition}[Strong compact dynamics]
    \label{def: strong compact dynamics}
    We say that a semiflow $F$ is {\BF s-uniformly continuous on $A\subset X$} if, 
    for every $T>0$, the map
    \[
    [0,T]\times A\longrightarrow X,
    \qquad
    (t,x)\longmapsto F^t(x),
    \]
    is uniformly continuous.
    Given an $\eps>0$ and a set $G\subset X$, we set 
    $$
    \text{\BF$W_\eps(G)$} = \bigcup_{t\geq0}F^t(N_\eps(G)).
    $$
    Notice that $W_\eps(G)$ is the smallest forward-invariant neighborhood of $N_\eps(G)$.
    
    We say that a global attractor $\cG$ is {\bf strong} if there is an $\varepsilon>0$ such that $\cG$ attracts $N_\varepsilon(\cG)$ and $F$ is $s$-uniformly continuous
    on $W_\varepsilon(\cG)$.
    If $F$ has a strong global attractor, we say that $F$ has {\bf strong compact dynamics}.
\end{definition}
\noindent
{\bf s-uniform continuity.}
The definition above is stronger than and supersedes the one introduced by the authors in~\cite{DLY26b}, which turns out to be not strong enough to prove some of the main results of that work. 
In turn, this modification also implicitly strengthens the definition of {\em strong attractor} and {\em strong compact dynamics} given in~\cite{DLY26b}.

\medskip
As pointed out above, 
every semiflow with compact dynamics over a locally compact space has strong compact dynamics.
In particular, when $X$ is compact, every semiflow has (strong) compact dynamics.
An important class of semiflows with strong compact dynamics in infinite dimension are the semiflows generated by dissipative reaction-diffusion PDEs (see~\cite{DLY26b} for more details).

{\BF
From this point on, we assume that $F$ has strong compact dynamics and we denote its (strong) global attractor by $\cG$.}

\section{Main results}
\label{sec: main results}
The following crucial result is a re-writing of Theorems~9 and~10 in~\cite{DLY26b}.
\begin{thmX}
    \label{thm: CR G}
    Let $F$ be a semiflow with strong compact dynamics. 
    Denote by $F_\cG$ the restriction of $F$ to its global attractor and by $\cR_\cC$ its Conley chain-recurrent set and by $\cR_{\cC_\cG}$ the Conley chain-recurrent set of $F_\cG$.
    Then:
    \begin{enumerate}
        \item $\cR_{\cC_{\cG}}=\cR_{\cC}$;
        \item $N$ is a node of $\cR_\cC$ if and only if it is a node of $\cR_{\cC_{\cG}}$;
        \item $\Gamma_{\cC_{\cG}}=\Gamma_{\cC}$.
    \end{enumerate}
\end{thmX}
%


Our first result is Theorem~\ref{thm: R_S subset G}, which is the analogue of Theorem~A for shadow chains. 
We start with 
four preparatory lemmas.

\begin{lemma}
    \label{lemma: Lyap}
    Let $\eps>0$ be small enough that $\cG$ attracts $N_\eps(\cG)$ and let $r\in(0,\eps)$.
    Then there exists an $a\in(0,\eps)$ such that
    $$
    F^t(N_a(\cG))\subset N_r(\cG)
    $$
    for all $t\geq0$,
    i.e. $\cG$ is Lyapunov stable.
\end{lemma}

\begin{proof}
        \medskip\noindent
    Since $\cG$ attracts $N_{\varepsilon}(\cG)$, there exists $T_r>0$ such that
    \[
    F^t\bigl(N_{\varepsilon}(\cG)\bigr)\subset N_r(\cG)
    \qquad\text{for all }t\geq T_r .
    \]
    Since $\cG$ is compact and invariant, continuity of $F$ on
    $[0,T_r]\times \cG$ implies that there exists $a\in(0,\varepsilon)$ such that
    \[
    F^t\bigl(N_a(\cG)\bigr)\subset N_r(\cG)
    \qquad\text{for all }t\in[0,T_r].
    \]
    Together with the previous inclusion, this implies that
    \[
    F^t\bigl(N_a(\cG)\bigr)\subset N_r(\cG)
    \qquad\text{for all }t\geq0.
\]
\end{proof}

\begin{lemma}
\label{prop: G is a trapping region for R_S}
For every $\rho>0$ small enough, there exists $\delta>0$ such that every $\delta$-chain \[ \gamma:[0,T]\to X \] with $\gamma(0)\in\cG$ satisfies \[ \gamma([0,T])\subset N_\rho(\cG).\] 
In particular, if $x\in \cG$ and $x\tto y$, then $y\in \cG$.
\end{lemma}


\begin{proof}
Let $\varepsilon_0>0$ be such that $G$ attracts $N_{\varepsilon_0}(G)$
and $F$ is uniformly continuous on
$
[0,1]\times W_{\varepsilon_0}(G).
$
Fix \(\rho\in(0,\varepsilon_0)\).

\medskip\noindent
{\bf Step 1: Shadow chains starting near \(\cG\) remain close to \(\cG\).}
By Lemma~\ref{lemma: Lyap}, there is an \(a\in(0,\rho)\)
such that
\[
F^t(N_a(\cG))\subset N_{\rho/3}(\cG)
\qquad\text{for all }t\ge 0 .
\]
Since \(\cG\) attracts \(N_{\varepsilon_0}(\cG)\), there is an integer \(m\ge 1\) such that
\[
F^m(N_{\varepsilon_0}(\cG))\subset N_{a/3}(\cG).
\]
Now, let
$
\eta<\min\{\rho/3,\,2a/3,\,\varepsilon_0-\rho/3\}.
$
Then
there exists $\delta>0$ such that
every restriction 
\[
        \gamma:[0,L]\to X,\qquad 0<L\le m,
\]
of a $\delta$-chain with $\gamma(0)\in N_a(G)$ satisfies
\[
        d(\gamma(t),F^t(\gamma(0)))<\eta
        \qquad\text{for all }t\in[0,L].
\]
Indeed, when $m=1$, this is just a direct consequence of the definition of shadow chain: for any $\delta$-chain with $\delta\in(0,\eta)$ we have
$$
d(\gamma(t),F^t(\gamma(0)))<\delta<\eta\quad\text{for every }t\in[0,1].
$$
Assume now that $m=2$ and let $\tau\in[0,1]$.
Then, by uniform continuity of $F$ on
$[0,1]\times W_{\varepsilon_0}(G)$, there is an $\alpha>0$ such that
\[
        d(u,v)<\alpha,\quad u,v\in W_{\varepsilon_0}(G)
        \quad\Longrightarrow\quad
        d(F^\tau(u),F^\tau(v))<\eta/2
\]
for all $\tau\in[0,1]$. 
Let now 
\(
        \delta\in(0,\min\{\alpha,\eta/2,\varepsilon_0-\rho/3\}).
\)
By the argument above,
\[
        d(\gamma(1),F^1(\gamma(0)))<\delta<\varepsilon_0-\rho/3
\]
Then, since $F^1(\gamma(0))\in N_{\rho/3}(G)$, both $\gamma(1)$ and
$F^1(\gamma(0))$ belong to $W_{\varepsilon_0}(G)$ and therefore, for every
$\tau\in[0,1]$ with $1+\tau\le L$,
\[
\begin{aligned}
d(\gamma(1+\tau),F^{1+\tau}(\gamma(0)))
&\le d(\gamma(1+\tau),F^\tau(\gamma(1))) + d(F^\tau(\gamma(1)),F^\tau(F^1(\gamma(0)))) \\
&< \delta+\eta/2
<\eta .
\end{aligned}
\]
Thus the desired estimate holds on $[0,L]$ when $m=2$.
By repeating this argument any finite number of times, we see that for any finite $m$ there is a $\delta>0$ such that 
\[
        d(\gamma(t),F^t(\gamma(0)))<\eta
        \qquad\text{for all }t\in[0,L].
\]


Hence, for every \(t\in[0,L]\),
\[
\operatorname{dist}(\gamma(t),\cG)
\le d(\gamma(t),F^t(\gamma(0)))+\operatorname{dist}(F^t(\gamma(0)),\cG)
<\eta+\rho/3<\rho,
\]
namely
\beqn
\gamma(t)\in N_\rho(\cG)
\qquad\text{for all }t\in[0,L].
\label{eq: 4}
\eeqn
If \(L=m\), then \(\gamma(0)\in N_a(\cG)\subset N_{\varepsilon_0}(\cG)\) and so
$
F^m(\gamma(0))\in N_{a/3}(\cG)
$
and
\[
\operatorname{dist}(\gamma(m),\cG)
\le d(\gamma(m),F^m(\gamma(0)))+\operatorname{dist}(F^m(\gamma(0)),\cG)
<\eta+a/3<a.
\]
Hence,
\beqn
\gamma(m)\in N_a(\cG).
\label{eq: 5}
\eeqn

\medskip\noindent
{\bf Step 2: Iterate the estimate.}
Let
$
        \gamma:[0,T]\to X
$
be a $\delta$-chain with $\gamma(0)\in \cG$. Since
$\cG\subset N_a(\cG)$, Step 1 applies to the restricted chain
\[
        \gamma|_{[0,\min\{m,T\}]} .
\]
Hence
\[
        \gamma(t)\in N_\rho(\cG)
        \qquad\text{for all }t\in[0,\min\{m,T\}].
\]
If $T\le m$, we are done. If $T>m$, then~\eqref{eq: 5} gives
\[
        \gamma(m)\in N_a(\cG).
\]
Therefore Step 1 applies again to the shifted shadow chain
\[
        s\longmapsto \gamma(m+s).
\]
Repeating this argument on the consecutive intervals
\[
        [0,m],\ [m,2m],\ [2m,3m],\ldots
\]
and then applying~\eqref{eq: 4} on the final interval of length at most $m$, we
obtain
\[
        \gamma(t)\in N_\rho(\cG)
        \qquad\text{for all }t\in[0,T].
\]
This proves the first assertion.

Now, suppose $x\in \cG$ and $x\tto y$. For the $\delta>0$ above,
choose a $\delta$-chain from $x$ to $y$. 
By the first assertion, its endpoint lies in $N_\rho(\cG)$. 
Since $\rho>0$ was arbitrary and $\cG$ is compact, it follows that $y\in \cG$.
\end{proof}

\begin{lemma}
\label{lem:shadow-recurrent-in-attractor}
Every shadow chain node of $F$ lies in $\cG$.
In particular,
\[
\cR_\cS\subset\cG.
\]
\end{lemma}

\begin{proof}
Let $x\in \cR_\cS$. 
By Lemma~\ref{lemma: OC1}, the orbit of $x$ is shadow chain-equivalent.
Since \(\cG\) attracts the compact set \(\{x\}\), we have
\[
d(F^t(x),\cG)\to0
\qquad\text{as }t\to\infty .
\]
Choose a \(t_n\to\infty\). 
Since \(\cG\) is compact, possibly after passing to a subsequence, there exists \(g\in \cG\) such that
\[
F^{t_n}(x)\to g .
\]
For every \(n\), the previous paragraph gives
\[
F^{t_n}(x)\tto x .
\]
Since \(\tto\) is closed, we obtain
\[
g\tto x .
\]
But \(g\in \cG\), so Lemma~\ref{prop: G is a trapping region for R_S} implies \(x\in \cG\).
Therefore \(\cR_\cS\subset \cG\). Since shadow nodes are equivalence classes inside
\(\cR_\cS\), every shadow node is contained in \(\cG\).
\end{proof}

\begin{lemma}
\label{lem:shadow-relation-restricts-to-attractor}
Denote by \(F_{\cG}\) the restriction of \(F\) to \(\cG\) and write $x\ttog y$ for the shadow chain relation of $F_\cG$. 
If \(x,y\in \cG\),
then
\[
x\tto y
\qquad\Longleftrightarrow\qquad
x\ttog y.
\]
\end{lemma}

\begin{proof}
The implication
\[
x\ttog y \Longrightarrow x\tto y
\]
is immediate, since every shadow chain for \(F_{\cG}\) is also a shadow chain
for \(F\).
Conversely, assume that \(x,y\in \cG\) with \(x\tto y\) and let
\(\varepsilon>0\). 
We show that there is an \(\varepsilon\)-chain in
\(\cG\) from \(x\) to \(y\).

Since \(\cG\) is strong, by Lemma~\ref{prop: G is a trapping region for R_S} we have that, for every \(\alpha>0\), every sufficiently accurate shadow chain starting in \(\cG\) remains in \(N_\alpha(\cG)\). 
Now, let $\varepsilon_0>0$ be such that $\cG$ attracts $N_{\varepsilon_0}(\cG)$ and
$F$ is $s$-uniformly continuous on $W_{\varepsilon_0}(\cG)$ and choose $\eta>0$ such that, whenever $u,v\in W_{\varepsilon_0}(\cG)$ and
$d(u,v)<\eta$, one has
\[
d(F^\tau(u),F^\tau(v))<\varepsilon
\qquad\text{for all }\tau\in[0,1];
\]
then choose $\nu>0$, with $\nu<\eta/4$, such that
\[
d(u,v)<\nu
\quad\Longrightarrow\quad
d(F^\tau(u),F^\tau(v))<\eta/4
\qquad\text{for all }\tau\in[0,1];
\]
finally, choose $\alpha>0$, with $\alpha<\min\{\nu,\eps_0\}$, such that
\[
d(u,v)<\alpha
\quad\Longrightarrow\quad
d(F^\tau(u),F^\tau(v))<\nu
\qquad\text{for all }\tau\in[0,1].
\]

Now, let 
$
\gamma:[0,T]\to X
$
be a shadow $\delta$-chain from $x$ to $y$ with $\delta\in(0,\min\{\nu,\eta/4\})$ small enough that 
\[
\gamma([0,T])\subset N_\alpha(\cG)
\]
and let
\[
0=t_0<t_1<\cdots<t_m=T
\]
be chosen so that
\[
1\leq t_{i+1}-t_i\leq2
\qquad\text{for all }i.
\]
For each $i$, choose $z_i\in \cG$ such that
\[
d(z_i,\gamma(t_i))<\alpha,
\]
with $z_0=x$ and $z_m=y$.
Set $h_i=t_{i+1}-t_i=1+\theta_i$, where $\theta_i\in[0,1]$. 
Then
\[
\begin{aligned}
d(F^{h_i}(z_i),z_{i+1})
&\leq
\phantom{+}d(F^{\theta_i}(F^1(z_i)),F^{\theta_i}(F^1(\gamma(t_i)))) \\
&\quad+
d(F^{\theta_i}(F^1(\gamma(t_i))),F^{\theta_i}(\gamma(t_i+1))) \\
&\quad+
d(F^{\theta_i}(\gamma(t_i+1)),\gamma(t_{i+1})) \\
&\quad+
d(\gamma(t_{i+1}),z_{i+1}).
\end{aligned}
\]
Since $\gamma$ is a $\delta$-chain and $F$ is s-uniformly continuous on $W_{\varepsilon_0}(\cG)$, each of those summands is smaller than $\eta/4$ and so
$$
d(F^{h_i}(z_i),z_{i+1})<\eta.
$$

Now define the concatenated curve $\Gamma:[0,T]\to \cG$ by
\[
        \Gamma(t)=F^{t-t_i}(z_i)
        \qquad\text{for }t\in[t_i,t_{i+1}),
\]
and set $\Gamma(T)=y=z_m$.

We claim that $\Gamma$ is an $\varepsilon$-chain in $\cG$ from
$x=z_0$ to $y=z_m$. 
Indeed, away from the partition times $t_i$, the shadow chain estimate is exact. 
If an interval $[t,t+\tau]$, with $\tau\in[0,1]$, crosses a partition time $t_{i+1}$, then the same
one-jump estimate used in the proof of Lemma~\ref{lemma: Concatenation} applies, with
\[
        d(F^{h_i}(z_i),z_{i+1})<\eta
\]
playing the role of the small error at the joining point. 
Since $\eta$ was chosen so that $\eta$-close points in $W_{\varepsilon_0}(\cG)$ remain
$\varepsilon$-close under $F^\tau$ for all $\tau\in[0,1]$, the crossing
estimate is also smaller than $\varepsilon$.

Thus, $\Gamma$ is an $\varepsilon$-chain in $\cG$ from $x$ to $y$.
Since $\varepsilon>0$ is arbitrary,
$
        x\ttog y .
$
\end{proof}

\begin{theorem}
\label{thm: R_S subset G}
Let \(F\) have strong compact dynamics, let \(\cG\) be its strong global
attractor, and let \(F_{\cG}=F|_{\cG}\). 
Denote by $\cR_{\cS_\cG}$ and $\Gamma_{\cS_\cG}$, respectively, the chain-recurrent set and the graph of $F_{\cG}$.
Then
\[
\cR_{\cS_\cG}=\cR_\cS,
\]
the shadow nodes of \(F_{\cG}\) and \(F\) are the same, and
\[
\Gamma_{\cS_\cG}=\Gamma_\cS.
\]
\end{theorem}

\begin{proof}
By Lemma~\ref{lem:shadow-recurrent-in-attractor},
$
\cR_\cS\subset \cG.
$
Since \(F_{\cG}\) is the restriction of \(F\) to \(\cG\),
Lemma~\ref{lem:shadow-relation-restricts-to-attractor} implies that, for
points of \(\cG\), the shadow relation for \(F\) is exactly the shadow relation
for \(F_{\cG}\). Therefore, a point of \(\cG\) is shadow recurrent for \(F\) if
and only if it is shadow recurrent for \(F_{\cG}\). Hence
$
\cR_{\cS_\cG}=\cR_\cS.
$

The same equivalence of relations on \(\cG\) also implies that two points of \(\cR_\cS\) are shadow-chain equivalent for \(F\) if and only if they are shadow chain equivalent for \(F_{\cG}\). 
Thus, the shadow nodes of \(F\) and
\(F_{\cG}\) coincide.

Finally, let \(M,N\) be two shadow nodes. Since \(M,N\subset \cG\), again by
Lemma~\ref{lem:shadow-relation-restricts-to-attractor},
\[
M\tto N
\qquad\Longleftrightarrow\qquad
M\ttog N.
\]
Thus the edges are the same, and therefore
$
\Gamma_{\cS_\cG}=\Gamma_\cS.
$
\end{proof}

By Theorem~\ref{thm: R_S subset G}, the shadow chain-recurrent structure (recurrent set, nodes, edges) of $F$ agrees with that of the restricted semiflow $F|_{\cG}$.
We already knew that that is the case also for the Conley chain-recurrent structure from Theorem~\ref{thm: CR G}.
Therefore, in order to prove the coincidence of the nodes  and edges of $\cS$ with the ones of $\cC$, it is enough to prove the claim for the restriction of $F$ on $\cG$.
From now on, we use the notation $x\otog y$ (resp. $x\ttog y$) to mean that $y$ is Conley chain-downstream (resp. shadow chain-downstream) from $x$ for $F|_\cG$.




\medskip\noindent
{\BF$\cR_\cC=\cR_\cS$.}
We now start working at our main result by showing that $\cR_\cC=\cR_\cS$.

\medskip\noindent
{\BF 1. $\cR_\cC\subset\cR_\cS$.}

%
\begin{lemma}
\label{lem:CtoP}
Let $x,y\in \cG$. Then $x\otog y \Rightarrow x\ttog y$.
\end{lemma}
\begin{proof}
Fix $\varepsilon>0$. Since $\cG$ is compact, there is a $\delta>0$ such that
$\delta<\varepsilon$ and, for $u,v\in\cG$,
\[
        d(u,v)<\delta
        \quad\Longrightarrow\quad
        d(F^\tau(u),F^\tau(v))<\varepsilon
        \qquad\text{for all }\tau\in[0,1].
\]
Since $x\otog y$, there is a $(\delta,1)$-chain
\[
        x=x_0,x_1,\ldots,x_n=y
\]
in $\cG$ with times $t_i\ge1$ such that
\[
        d(F^{t_i}(x_i),x_{i+1})<\delta,
        \qquad i=0,\ldots,n-1 .
\]
Let $T_{-1}=0$ and $T_i=\sum_{j=0}^i t_j$ and define
\[
        \gamma(t)=F^{t-T_{i-1}}(x_i)
        \qquad\text{for }t\in[T_{i-1},T_i),
\]
and finally set $\gamma(T_{n-1})=y$.

We claim that $\gamma$ is an $\varepsilon$-chain from $x$ to $y$.
Indeed, away from the jump times, the shadow chain estimate is exact.
At a jump time, the same one-jump estimate used in the proof of
Lemma~\ref{lemma: Concatenation} applies, with the jump
\[
        d(F^{t_i}(x_i),x_{i+1})<\delta
\]
playing the role of the small error at the joining point. Since
$\delta$ was chosen so that such errors remain smaller than $\varepsilon$ under
$F^\tau$ for every $\tau\in[0,1]$, the crossing estimate is also
smaller than $\varepsilon$. 
Since $\eps$ was arbitrary, this shows that $x\ttog y$.
\end{proof}

\begin{proposition}
    \label{prop: R_C subset R_S}
    Let $F$ have strong compact dynamics.
    Then $\cR_\cC\subset\cR_\cS$.
\end{proposition}
\begin{proof}
Let $x\in \cR_\cC$. 
By Theorem~\ref{thm: CR G}, $x\in\cG$ and there exists $t>0$ such that $x\otog F^t(x)$ and $F^t(x)\otog x$.
By Lemma~\ref{lem:CtoP}, $x\ttog F^t(x)$ and $F^t(x)\ttog x$, namely $x$ is shadow chain-equivalent to $F^t(x)$ in $\cG$.
Hence, in particular, $x$ is shadow chain-equivalent to $F^t(x)$ in $X$ and so $x\in \cR_\cS$. 
Thus, $\cR_\cC\subset \cR_\cS$.
\end{proof}

\medskip\noindent
{\BF 2. $\cR_\cS\subset\cR_\cC$.}

\medskip
The following results about Conley chains in a compact setting are crucial for the remainder of the article. 
\begin{thmX}[Ayala {\em et al.}, 2006~\cite{Aya06}, Proposition~5.5]
    \label{thm: Ayala}
    Let $y\in \cR_\cC$, $x\in X$, and $\tau>0$ and assume that $X$ is compact and that, for every $\varepsilon>0$, there exists an $(\varepsilon,\tau)$-chain from $x$ to $y$. 
    Then $x\oto y$.
\end{thmX}
In fact, one may choose the resulting $(\varepsilon,T)$-chains with all jump times equal to $T$.
Notice that, although stated originally for flows, the proof of the result above uses only positive times and therefore applies verbatim to semiflows.
\begin{thmX}[Hurley, 1995~\cite{Hur95}, Section~3]
    \label{thm: Hur}
    Let \(F\) be a continuous-time semiflow on a compact metric space \(X\) and let \(\tau>0\). 
    Then \(x\in \cR_\cC\) if and only if, for every \(\varepsilon>0\), there exists an \((\varepsilon,\tau)\)-chain from \(x\) to itself.
\end{thmX}



\begin{lemma}
    \label{lem: eps,1 chain}
    Let $x,y\in\cG$ with $x\ttog y$.
    Then, for every $\varepsilon>0$, there is an $(\varepsilon,1)$-chain in $\cG$ from $x$ to $y$.
\end{lemma}

\begin{proof}
Fix $\varepsilon>0$. Since $\cG$ is compact, $F$ is uniformly continuous on
$[0,1]\times \cG$. 
Hence there exists a $\delta_0>0$ such that
\[
d(u,v)<\delta_0
\quad\Longrightarrow\quad
d\bigl(F^\tau(u),F^\tau(v)\bigr)<\varepsilon/2
\]
for all $\tau\in[0,1]$ and all $u,v\in \cG$. 
Set
\[
\delta:=\min\{\delta_0,\varepsilon/2\}.
\]
Since $x\ttog y$, there exists a shadow $\delta$-chain
\[
\gamma:[0,L]\to \cG
\]
from $x$ to $y$. 
Let $N\geq1$ be such that
\[
1\leq h:=\frac{L}{N}\leq2.
\]
For instance, one may take $N=\lceil L/2\rceil$. Set
\[
s_i=ih,
\qquad
c_i=\gamma(s_i),
\qquad
i=0,\ldots,N.
\]
Then $c_0=x$ and $c_N=y$. Write
\[
h=1+\theta,
\qquad
\theta\in[0,1].
\]

For each $i=0,\ldots,N-1$, the shadow-chain condition gives
\[
d\bigl(\gamma(s_i+1),F^1(\gamma(s_i))\bigr)<\delta
\]
and
\[
d\bigl(\gamma(s_i+1+\theta),F^\theta(\gamma(s_i+1))\bigr)<\delta.
\]
Therefore,
\[
\begin{aligned}
d\bigl(c_{i+1},F^h(c_i)\bigr)
&=
d\bigl(\gamma(s_i+1+\theta),F^\theta(F^1(\gamma(s_i)))\bigr) \\
&\leq
d\bigl(\gamma(s_i+1+\theta),F^\theta(\gamma(s_i+1))\bigr) \\
&\quad+
d\bigl(F^\theta(\gamma(s_i+1)),
F^\theta(F^1(\gamma(s_i)))\bigr) \\
&<\delta+\varepsilon/2
\leq\varepsilon.
\end{aligned}
\]
Since $h\geq1$, the sequence
\[
c_0,c_1,\ldots,c_N
\]
with all times equal to $h$ is an $(\varepsilon,1)$-chain in $\cG$ from $x$ to
$y$.
\end{proof}

\begin{proposition}
\label{prop: RS subset RC}
Let $F$ have strong compact dynamics.
Then $\cR_\cS\subset \cR_\cC$. 
Moreover, if $x\in\cG$, $y\in \cR_\cC$ and
$x\tto y$, then $x\oto y$.
\end{proposition}

\begin{proof}
Let $x\in \cR_\cS$. 
By Theorem~\ref{thm: R_S subset G}, $x$ is shadow chain-recurrent for $F_\cG$ as well and so there exist a $t>0$ such that
\[
x\ttog F^t(x)
\qquad\text{and}\qquad
F^t(x)\ttog x.
\]
Since $\ttog$ is transitive, $x\ttog x$. 
By the previous lemma, for every $\varepsilon>0$ there exists an
$(\varepsilon,1)$-Conley chain in $\cG$ from $x$ to itself. 
By Theorem~\ref{thm: Hur} applied to $F_\cG$, this implies that $x$ is chain-recurrent for $F_\cG$;
by Theorem~\ref{thm: CR G}, this means that $x\in \cR_\cC$.
Thus
\[
\cR_\cS\subset \cR_\cC.
\]

Now assume $x\in\cG$, $y\in \cR_\cC$ and $x\tto y$.
Then, by Theorem~\ref{thm: CR G}, $y\in\cG$ and, by Lemma~\ref{lem:shadow-relation-restricts-to-attractor}, $x\ttog y$.
Hence, by Lemma~\ref{lem: eps,1 chain}, for every $\varepsilon>0$ there exists an $(\varepsilon,1)$-Conley chain from $x$ to $y$.
Since $y\in \cR_\cC$, by Theorem~\ref{thm: Ayala} applied to $F_\cG$ it follows that $x\otog y$. Therefore, $x\oto y$.
\end{proof}

We are now finally ready to prove our main (and last) result.

\begin{theorem}
    \label{thm:main}
    Assume that $F$ has strong compact dynamics. 
    Then
    the following holds:
    \begin{enumerate}
        \item $\cR_{\cC}=\cR_{\cS}$, so we denote this set simply by $\cR$.
        \item If $x,y\in\cR$, then $x\oto y$ if and only if $x\tto y$.
        \item $M$ is a Conley node if and only if it is a shadow node.
        \item $\Gamma_\cC=\Gamma_\cS$.
    \end{enumerate}
\end{theorem}

\begin{proof}
\noindent
\textbf{(1) Equality of recurrent sets.}
This is an immediate consequence of Propositions~\ref{prop: R_C subset R_S} and~\ref{prop: RS subset RC}.

\medskip\noindent
\textbf{(2) Downstream equivalence for recurrent endpoints.}
Let $x,y\in \cR$. 

Assume first that $x\oto y$.
Let \(M,N\) be the Conley nodes containing \(x\) and \(y\). 
Then \(M\oto N\). 
By Theorem~\ref{thm: CR G}, the same edge occurs for $F_\cG$. 
Since the Conley nodes of \(F\) and \(F_\cG\) coincide, and points inside each node are mutually Conley chain-equivalent, it follows
that \(x\otog y\). 
Then, by Lemma~\ref{lem:CtoP}, we have that \(x\ttog y\), and so
\(x\tto y\).

Assume now that $x\tto y$. 
Then, since $y\in\cR=\cR_{\cS}$, Proposition~\ref{prop: RS subset RC} implies that $x\oto y$.

\medskip\noindent
\textbf{(3),(4).}
These are immediate consequences of items (1) and (2).
\end{proof}

\medskip\noindent
{\bf The compact dynamics assumption cannot be dropped.}
Indeed, the following slight modification of Example~\ref{ex: C neq S} shows that $\cR_\cC$ and $\cR_\cS$ can be different for semiflows that do not have compact dynamics.

\begin{example}
Let
\[
A=\{(s,0):s\in\mathbb R\},\qquad
B=\{(s,e^{-|s|}):s\in\mathbb R\},
\]
and set
\[
X=A\cup B\subset\mathbb R^2
\]
with the metric induced by the Euclidean distance. Write
\[
a_s=(s,0),\qquad b_s=(s,e^{-|s|}).
\]
Define a flow \(F\) on \(X\) by
\[
F^t(a_s)=a_{s+t},\qquad F^t(b_s)=b_{s-t},
\qquad s,t\in\mathbb R.
\]
Thus points on \(A\) move to the right and points on \(B\) move to the left.
The two components of \(X\) are asymptotic both as \(s\to+\infty\) and as
\(s\to-\infty\). The flow has no compact dynamics, since the positive orbit
\(\{a_t:t\ge0\}\) escapes every compact subset of \(X\).

We first show that \(a_0\in \cR_\cC\). It is enough to show that \(a_0\) is
Conley chain-equivalent to \(a_1=F^1(a_0)\). Let \(\varepsilon>0\) and
\(T>0\). Choose \(R>T+1\) so large that \(e^{-R}<\varepsilon\). Then
\[
a_0,\ b_R,\ a_{-R},\ a_1
\]
with transition times \(R,2R,R+1\) is a Conley \((\varepsilon,T)\)-chain
from \(a_0\) to \(a_1\), because
\[
d(F^R(a_0),b_R)=d(a_R,b_R)=e^{-R}<\varepsilon,
\]
\[
d(F^{2R}(b_R),a_{-R})=d(b_{-R},a_{-R})=e^{-R}<\varepsilon,
\]
and
\[
F^{R+1}(a_{-R})=a_1.
\]
Similarly,
\[
a_1,\ b_R,\ a_{-R},\ a_0
\]
with transition times \(R-1,2R,R\) is a Conley \((\varepsilon,T)\)-chain
from \(a_1\) to \(a_0\). Hence \(a_0\) and \(a_1\) are Conley
chain-equivalent, and so \(a_0\in \cR_\cC\).

We now prove that \(a_0\notin \cR_\cS\). 
It is enough to show that, for every \(t>0\), one does not have \(a_t\tto a_0\). Fix \(t>0\), and choose
\[
0<\varepsilon<\min\left\{\frac{t}{4},\frac14\right\}
\]
so small that
\[
-\log\varepsilon>\frac{7}{2}\varepsilon .
\]
We prove that there is no \(\varepsilon\)-chain from \(a_t\) to \(a_0\).

Suppose, by contradiction, that
\[
\gamma:[0,L]\to X
\]
is an \(\varepsilon\)-chain from \(a_t\) to \(a_0\). If \(\gamma\)
never jumps from one branch to the other, then \(\gamma([0,L])\subset A\).
Writing
\[
\gamma(r)=a_{\sigma(r)}
\]
and taking a partition
\[
0=r_0<r_1<\cdots<r_m=L
\]
with \(r_i-r_{i-1}\le1\) and \(m\le L+1\), the shadow estimate gives
\[
|\sigma(r_i)-\sigma(r_{i-1})-(r_i-r_{i-1})|<\varepsilon
\]
for each \(i=1,\ldots,m\). Summing,
\[
\sigma(L)> \sigma(0)+L-m\varepsilon
\ge t+L-(L+1)\varepsilon>0,
\]
because \(L\ge1\) and \(\varepsilon<1/4\). This contradicts
\(\gamma(L)=a_0\). Hence \(\gamma\) must jump between the two branches.

Since \(\gamma\) is piecewise continuous, it has only finitely many branch jumps.
Let \(\theta\) be the last jump time between \(A\) and \(B\). Since
\(\gamma(L)=a_0\in A\), this last jump must be from \(B\) to \(A\). Write
\[
\lim_{r\to\theta^-}\gamma(r)=b_\beta,\qquad
\lim_{r\to\theta^+}\gamma(r)=a_\alpha .
\]
Taking \(r\to\theta^-\) and \(s\to\theta^+\) in the shadow estimate gives
\[
d(b_\beta,a_\alpha)\le\varepsilon.
\]
Therefore
\[
|\alpha-\beta|\le\varepsilon
\]
and, since the vertical distance from \(b_\beta\) to \(A\) is \(e^{-|\beta|}\),
\[
e^{-|\beta|}\le\varepsilon.
\]
Thus
\beqn
\label{eq: -2}
|\beta|\ge-\log\varepsilon.
\eeqn

We now show that the last jump must occur very close to the final time. 
Let
\[
u\in(0,\min\{1,L-\theta\}).
\]
Since there are no branch jumps after \(\theta\), the point \(\gamma(\theta+u)\) lies on \(A\). 
Applying the shadow estimate from times just before \(\theta\) to \(\theta+u\), and then letting the initial time tend to \(\theta^-\), gives
\[
d\bigl(\gamma(\theta+u),b_{\beta-u}\bigr)\le\varepsilon.
\]
Applying the shadow estimate from times just after \(\theta\) to \(\theta+u\), and then letting the initial time tend to \(\theta^+\), gives
\[
d\bigl(\gamma(\theta+u),a_{\alpha+u}\bigr)\le\varepsilon.
\]
Hence
\[
|\alpha+u-(\beta-u)|\le2\varepsilon.
\]
Together with \(|\alpha-\beta|\le\varepsilon\), this implies
\[
2u\le3\varepsilon.
\]
Since this holds for every \(u\in(0,\min\{1,L-\theta\})\), we get
\beqn
\label{eq: -1}
L-\theta\le \frac32\varepsilon.
\eeqn

By~\eqref{eq: -1}, \(L-\theta<1\). 
Applying the shadow estimate from times just after \(\theta\) to the final time \(L\), and using \(\gamma(L)=a_0\), gives
\[
d\bigl(a_0,F^{L-\theta}(a_\alpha)\bigr)\le\varepsilon.
\]
Equivalently,
\[
|\alpha+(L-\theta)|\le\varepsilon.
\]
Thus, using~\eqref{eq: -1},
\[
|\alpha|\le L-\theta+\varepsilon\le \frac52\varepsilon.
\]
Since \(|\alpha-\beta|\le\varepsilon\), it follows that
\[
|\beta|\le |\alpha|+\varepsilon\le \frac72\varepsilon,
\]
contradicting~\eqref{eq: -2} and the choice of \(\varepsilon\). Therefore, there is no
\(\varepsilon\)-chain from \(a_t\) to \(a_0\).

Since \(t>0\) was arbitrary, \(a_0\) is not shadow chain-equivalent to any positive-time image \(F^t(a_0)\). 

Hence
\(
a_0\notin \cR_\cS
\)
and so 
$$
\cR_\cC\neq\cR_\cS.
$$
\end{example}

\medskip\noindent
{\bf Final remark.}
In~\cite{Hur91}, Hurley introduced a new type of $(\eps,T)$-chains in case of non-compact spaces where $\eps$, rather than being a positive constant, is a positive continuous function.
We refer to those chains as {\em Hurley chains}.
Hurley chains coincide with Conley chains in the compact setting and it is shown in~\cite{Hur91} that several important results that hold for Conley chains on compact spaces fail to hold on non-compact spaces but do hold in both settings for Hurley chains.

We did not consider streams of Hurley chains in our works because our primary interest are semiflows with strong compact dynamics but we conjecture here that a ``Hurley version'' of our shadow chains would give rise to the same chain stream as the corresponding Hurley chain stream.

\section*{Acknowledgments}
This material is based upon work supported by the National Science Foundation under Grant DMS-2308225.

\bibliographystyle{unsrt}
\bibliography{refs.bib} 

\end{document}